\documentclass[12pt]{amsart}
\usepackage{Baskervaldx}
\usepackage[baskervaldx,cmintegrals,bigdelims,vvarbb]{newtxmath}
\usepackage[cal=cm]{mathalfa}
\usepackage{enumerate, geometry, pgfplots, graphicx, subcaption, amsmath, array, pdfsync, tikz, tikz-cd, color, url, float, wrapfig, comment, MnSymbol}

\usetikzlibrary{positioning}

\usepackage[hyperfootnotes=false, pdfencoding=auto]{hyperref}
\hypersetup{
  colorlinks= true,
  linkcolor=[RGB]{0,89,42},
  urlcolor=[RGB]{0,89,42},
  citecolor=[RGB]{0,89,42}
}

\usepackage[font=footnotesize]{caption}

\usetikzlibrary{calc}

\numberwithin{equation}{section}

\title{Higher index Fano varieties with finitely many birational automorphisms}
\author{Nathan Chen and David Stapleton}

\definecolor{mycolor}{RGB}{146, 214, 203}
\definecolor{myothercolor}{RGB}{179, 215, 232}

\newtheorem{theorem}{Theorem}

\newtheorem{proposition}[theorem]{Proposition}
\numberwithin{theorem}{section}

\newtheorem*{thmA}{Theorem A}
\newtheorem*{thmB}{Theorem B}
\newtheorem*{corC}{Corollary C}

\theoremstyle{definition}

\newcommand{\hr}[2]{\hyperref[#1]{#2}}

\theoremstyle{definition}
\newtheorem{remark}[theorem]{Remark}

\newtheorem*{question*}{Question}

\newtheorem{example}[theorem]{Example}
\newtheorem{definition}[theorem]{Definition}

\def\CC{{\mathbb C}}
\def\LL{{\mathbb L}}

\def\PP{{\mathbb{P}}}

\def\Pic{{\mathrm{Pic}}}
\def\Cl{{\mathrm{Cl}}}
\def\Spec{{\mathrm{Spec}}}

\def\Oc{{\mathcal{O}}}

\def\Lc{{\mathcal{L}}}

\def\xs{{x_1,\ldots,x_n}}

\def\Bir{{\mathrm{Bir}}}
\def\id{{\mathrm{id}}}

\def\mut{{\tilde{\mu}}}

\def\dra{{\dashrightarrow}}
\def\ra{{\rightarrow}}

\def\Xt{{\widetilde{X}}}
\def\gt{{\tilde{g}}}
\def\Lt{{\widetilde{\LL}}}

\def\cl{{\colon}}

\def\Hom{{\mathrm{Hom}}}

\def\PGL{{\mathrm{PGL}}}
\def\GL{{\mathrm{GL}}}
\def\Aut{{\mathrm{Aut}}}

\def\BirPic{{\mathrm{Pic}_{\Bir(X)}}}

\newcommand*{\sheafTor}{\mathcal{T}\kern -.5pt or}
\newcommand*{\Extc}{\mathcal{E}\kern -.5pt xt}

\pgfplotsset{compat=1.15}

\newcommand{\mybigwedge}{\raisebox{.25ex}{\scalebox{0.86}{$\bigwedge$}}}

\thanks{The first author's research was partially supported by an NSF postdoctoral fellowship, DMS-2103099. The second author's research is partially supported by the NSF FRG grant number 1952399.}

\begin{document}
\maketitle

\thispagestyle{empty}

A famous problem in birational geometry is to determine when the birational automorphism group of a Fano variety is finite. The Noether-Fano method has been the main approach to this problem. The purpose of this paper is to give a new approach to the problem by showing that there are Fano varieties in all positive characteristics that carry an ample and birationally equivariant line bundle. As a consequence these Fano varieties have finite birational automorphism groups.

Recall that a \textit{Fano variety} $X$ is a variety with mild (at worst klt) singularities such that $-K_X$ is ample. The \textit{index}\footnote{One could alternatively define the Fano index by only considering Cartier divisors. At least when $n\ge 4$, our varieties are factorial so the definitions are equivalent.} of $X$ is the largest number $r$ such that $-K_X \equiv rH$ for an ample Weil divisor $H$. Iskovskikh and Manin \cite{IskMan71} showed that the birational automorphism group, $\Bir(X)$, of a smooth quartic threefold is finite, which implied that a smooth quartic threefold is not rational. Given a rational map of Fano varieties, their approach relied on a detailed study of singularities of divisors in the corresponding linear series. This approach is now referred to as the \textit{Noether-Fano} method (using ideas of Fano \cite{Fano1908,Fano1915}). An immense amount of work has been devoted to proving the birational rigidity, and thus finiteness, of $\Bir(X)$ for other Fano varieties of index one. From the contributions of many authors (cf. \cite{Cheltsov2000,Corti95,deFernex13,deFernex16Erratum,dFEM03,Pukhlikov87,Pukhlikov02}) we now know that in characteristic zero any smooth Fano hypersurface $X \subset \CC\PP^{n+1}$ of degree $n+1$ is superrigid. As a consequence, these index one Fano hypersurfaces have finite birational automorphism groups. For index two Fano varieties, Pukhlikov has a number of results on finiteness of birational automorphisms \cite{Pukhlikov10, Pukhlikov16, Pukhlikov20, Pukhlikov21}. These results classify rationally connected fibrations of these varieties with base of dimension one. Little is known about $\Bir(X)$ for higher index Fano varieties. Here we give the first examples of higher index Fano varieties with finite birational automorphism groups:

\begin{thmA}\label{thm:A}
For every characteristic $p>0$ there are Fano varieties of arbitrarily large index over a field of characteristic $p$ with trivial birational automorphism groups.
\end{thmA}

In positive characteristic, Koll\'{a}r \cite{Kollar95} observed that certain Fano varieties that are $p$-cyclic covers in characteristic $p$ carry global differential forms, and proved that a very general hypersurface $X \subset \CC\PP^{n+1}$ of degree at least $2\lceil (n+3)/3 \rceil$ is not rational. These forms have also been used to show that Fano hypersurfaces of high degree are far from being rational in other ways. For example, Totaro \cite{Totaro} used these forms to prove that hypersurfaces in a slightly larger range are not even stably rational. Using unramified cohomology as an obstruction, Schreieder \cite{Schreieder19} improved these results and showed that a very general hypersurface of degree $d\ge \log_2(n)+2$ is not stably rational. The arguments of Totaro and Schreieder both involve the specialization property of decomposition of the diagonal developed by Colliot-Th\'el\`ene and Voisin. In other degree ranges, by studying the positivity properties of these forms in more detail, the authors demonstrated that the degrees of irrationality of complex Fano hypersurfaces can be arbitrarily large, and -- in a different range -- the degrees of possible rational endomorphisms on complex Fano hypersurfaces  must satisfy certain congruence conditions (see (\cite{FanoIrrat,RatEnd}).

We work with the $p$-cyclic covers that Koll\'ar used:
\[
\mu\cl Y\ra X.
\]
They have mild (terminal) isolated singularities and admit a straightforward resolution of singularities:
\[
\sigma\cl Z \ra Y.
\]
An important step in proving Theorem A is the computation of the space of global $(n-1)$-forms on $Z$. In doing so, we show that the only global $(n-1)$-forms are the forms that Koll\'ar found.

\begin{thmB}
Let $k$ be an algebraically closed field of characteristic $p>0$. Let $n\ge 3$ (if p=2 then assume $n$ is even). Let $X\subset \PP^{n+1}_k$ be a smooth degree $e$ hypersurface and fix $d>0$. If $Y$ is a $p$-cyclic cover over a general section of $\Oc_X(pd)$ with resolution $Z$ and
\[
(p-1)d\le n-e\le pd-3,
\]
then $H^0(Z,\mybigwedge^{n-1}\Omega_Z) \cong H^0(X,\omega_X(pd))$.
\end{thmB}

\noindent The first inequality implies that $Y$ is Fano of index at least 2. The second inequality implies that $\omega_X(pd)$ is very ample.

We introduce the notion of \textit{birational equivariance} for line bundles, which arises naturally in this setting. We show that having a nontrivial birationally equivariant line bundle is very strong. Importantly, the global sections of a birationally equivariant line bundle $\Lc$ on $Y$ are naturally a representation of $\Bir(Y)$. This allows us to show:

\begin{corC}
In the setting of Theorem B,  $\mu^*(\omega_X(pd))$ is an ample and birationally equivariant line bundle on $Y$ and there is an injection $\Bir(Y)\cong \Aut(Y) \hookrightarrow \Aut(X)$.
\end{corC}

\noindent For an alternative perspective on these results, in \cite[V.5.20]{Kollar96} Koll\'ar views the map $Y\ra X$ as a birational invariant of $Y$. 

These theorems are proved in slightly greater generality and apply to other $p$-cyclic covers with appropriate hypotheses. The results lead us to ask the following question:

\begin{question*}
Can finiteness of birational automorphisms of Fano varieties in characteristic $p$ be used to prove that complex Fano varieties have finitely many birational automorphisms?
\end{question*}

\noindent \textbf{Outline and conventions:} We begin \S 1 by defining birational equivariance for line bundles $\Lc$ and
giving some properties and relevant examples. In \S 2, we describe the resolutions of certain $p$-cyclic covers and check that they have terminal singularities. Finally, in \S 3 we prove Theorem B by computing the $(n-1)$-forms on the cyclic covers, which leads to a proof of Theorem A and Corollary C.

Throughout we work over an algebraically closed field $k$. A variety is an integral $k$-scheme of finite type. We do not give the birational automorphism group any scheme structure.

\noindent{\textbf{Acknowledgments.}} We would like to thank J\'anos Koll\'ar, Mircea Musta\c{t}\u{a}, Alex Perry, Aleksandr Pukhlikov, and Burt Totaro for valuable comments and discussions.

\section{Birationally equivariant line bundles}

The goal of this section is to introduce the notion of a birationally equivariant line bundle, to give some examples, to state some basic properties, and explain how they can be used to study the birational automorphism group. We consider $\Bir(X)$ as an abstract group, without any scheme structure.

Let $k$ be an algebraically closed field, and let $X$ be a normal algebraic variety over $k$. By variety, we mean a projective integral scheme of finite type over $k$. Let
\[
f \cl X\dra X
\]
be a rational endomorphism. The map $f$ is defined on some open set $i\cl U\hookrightarrow X$ such that $X\setminus U$ has codimension 2 in $X$. To start we define the pullback of a line bundle along $f$.

\begin{definition}
Let $\Lc$ be a line bundle on $X$. The \textit{pullback of $\Lc$ along} $f$ is defined by
(1) first pulling back to a line bundle $f^*(\Lc)\in \Pic(U)$, and then (2) pushing forward $i_*(f^*(\Lc))$ to get a reflexive rank 1 sheaf. This gives a group homomorphism:
\[
f^{\ast} \cl \Pic(X) \rightarrow \Cl(X),
\]
(where we identify the divisor class group with the group of reflexive rank 1 sheaves with reflexive tensor product).
\end{definition}

\begin{definition}
We say that $\Lc$ is \textit{birationally equivariant} if for every $g\in \Bir(X)$ there is a choice of an isomorphism:
\[
\phi_g\cl g^*\Lc\ra \Lc.
\]
subject to the following compatibility condition: for all $g_1, g_2 \in \Bir(X)$,
\begin{center}
\begin{tikzcd}
(g_1\cdot g_2)^*(\Lc) \arrow[r, equals] &[-0.3in] g_2^*(g_1^*\Lc) \arrow[dr,swap,"\phi_{g_1 \cdot g_2}"]\arrow[r,"g_2^*(\phi_{g_1})"]&g_2^*(L).\arrow[d,"\phi_{g_2}"]\\
& &\Lc
\end{tikzcd}
\end{center}
\end{definition}

\begin{remark}
It also makes sense to talk about $G$-birationally equivariant line bundles for any subgroup $G\le \Bir(X)$, any group homomorphism $G\ra \Bir(X)$, as well as for vector bundles on $X$.
\end{remark}

\begin{example}
The \textit{Cremona involution} $\tau \cl \PP^n \dra \PP^n$ is defined by
\begin{align*}
\tau([x_0\cl \cdots \cl x_n])&=[1/x_0\cl \cdots \cl 1/x_n]\\
&=[x_1\cdots x_n:x_0 x_2\cdots x_n: \cdots : x_0\cdots x_{n-1}].
\end{align*}
The pullback $\tau^*\cl \Pic(\PP^n)\ra \Pic(\PP^n)$ is multiplication by $n$. However, $\tau^{\circ 2} = \id\in \Bir(\PP^n)$. Thus even if $\Pic(X) = \Cl(X)$, the map
\[
\Bir(X)\ra \Hom(\Pic(X),\Pic(X))
\]
is only a map of sets -- it does not always respect composition. This also shows that $\PP^n$ does not admit any nontrivial birationally equivariant line bundles when $n\ge 2$.
\end{example}

\begin{theorem}[Basic Properties of Birationally Equivariant Line Bundles]\label{PropertiesBirThm}\
\begin{enumerate}
\item If $\Lc_1$ and $\Lc_2$ are birationally equivariant line bundles on $X$, then so is $\Lc_1\otimes \Lc_2$. 
\item Likewise, the inverse of a birationally equivariant line bundle is naturally birationally equivariant. In particular,
\[
\BirPic(X):=\{\text{line bundles with birational equivariant structure}\}
\]
is a group under tensor product and the forgetful map
\[
\BirPic(X)\ra \Pic(X)
\]
is a group homomorphism with kernel equal to the group of 1-dimensional $k$-representations of $\Bir(X)$ under tensor product.
\item Let $\mu \cl \Xt \ra X$ be a proper birational morphism. If $\Lc$ is a line bundle on $X$ and $\mu^{\ast}\Lc$ has a birationally equivariant structure, then $\Lc$ is naturally birationally equivariant.
\item If $\Lc$ is a birationally equivariant line bundle on $X$ and $H^0(X,\Lc)\ne 0$, then there is a representation $\rho\cl \Bir(X)\ra \GL(H^0(X,\Lc)^\vee)$ such that the following diagram commutes:
\[
\begin{tikzcd}
X\arrow[r,dashed,"g"]\arrow[d,"\pi"]&X\arrow[d,"\pi"]\\
\PP(H^0(X,\Lc)^\vee)\arrow[r,"\rho(g)"]&\PP(H^0(X,\Lc)^\vee).
\end{tikzcd}
\]
\item In the setting of (4), let $Z\subset \PGL(H^0(\Lc)^\vee)$ be the closure of the image of $X$. For all $g\in \Bir(X)$, $\rho(g)$ restricts to an automorphism of $Z$, which induces a group homorphism:
\[
\Bir(X)\ra \Aut(Z)
\]
and the kernel consists of $g\in \Bir(X)$ such that $\pi\circ g = \pi.$
\item In the setting of (4), if there is a nonempty open set $U\subset X$ such that $\pi|_U$ is injective on the $k$-points of $U$ (e.g. if $\pi$ is birational or generically finite and purely inseparable), then the homomorphism $\Bir(X)\ra \Aut(Z)$ is injective.
\item If $X$ has an ample birationally equivariant line bundle, then $\Bir(X)\cong \Aut(X)$.
\end{enumerate}
\end{theorem}

\begin{proof}
For (1), to give the tensor product $\Lc_1\otimes \Lc_2$ a birationally equivariant structure, one may assign:
\[
\phi_{g} \cl g^*(\Lc_1\otimes \Lc_2) \cong g^{\ast}\Lc_1 \otimes g^{\ast}\Lc_2 \xrightarrow{(\phi_1)_{g}\otimes (\phi_2)_g} \Lc_1\otimes \Lc_2.
\]
Note that the first isomorphism is canonical. The compatibility condition is easy to check.

In (2), if $\Lc$ is birationally equivariant with isomorphisms $\phi_g$, then there are isomorphisms:
\[
\phi'_g = (\phi_g^\vee)^{-1}\cl g^*(\Lc^\vee)\ra \Lc^\vee.
\]
Compatibility is easy to check. It is clear that the map
\[
\BirPic(X)\ra \Pic(X)
\]
is a group homomorphism. The kernel is given by equivariant structures on the trivial line bundle. These give rise to 1-dimensional representations on $H^0(X,\Oc_X)$ which determine the birationally equivariant bundle up to isomorphism.

To prove (3), let $g\in \Bir(X)$ let $\gt\in \Bir(\Xt)$ be the corresponding birational automorphism. Assume that both $g$ and $\gt$ are defined away from codimension 2. Let $U\subset X$ be an open set so that (a) $X\setminus U$ has codimension at least 2 in $X$, (b) $\pi^{-1}$ is defined on $U$, (c) $g$ is defined on $U$, and (d) $\gt$ is defined on $\pi^{-1}(U)$. By assumption, there is an isomorphism:
\[
\phi_{\gt}|_{\mu^{-1}(U)} \cl (\gt^*\mu^*\Lc)|_{\mu^{-1}(U)} \ra (\mu^*\Lc)|_{\mu^{-1}(U)}.
\]
This gives an isomorphism:
\[
\phi_g|_U \cl g^*(\Lc)|_U\ra \Lc|_U
\]
which uniquely extends to an isomorphism:
\[
\phi_g \cl g^*(\Lc)\ra \Lc.
\]
Lastly, compatibility follows as it can be checked on any nonempty open set (such as $U$).

For (4), the isomorphisms $\phi_g$ give rise to isomorphisms of global sections:
\[
H^0(X,L)\xrightarrow{g^*}H^0(X,g^*\Lc)\xrightarrow{\phi_g} H^0(X,L).
\]
Let $\rho(g)^{\vee}$ denote the composition. The compatibility implies that the dual isomorphisms satisfy:
\[
\rho(g_1\cdot g_2)=\rho(g_1)\cdot \rho(g_2)\in \GL(H^0(X,L)^\vee).
\]
Commutativity of the diagram follows from the fact that for a general $x\in X$, the map $\rho(g)^\vee$ gives an isomorphism between sections of $H^0(X,\Lc)$ vanishing at $x$ and those vanishing at $g(x)$.

To prove (5), it suffices to observe that the matrix $\rho(g)$ preserves the closure of the image $\pi(X)$, which is clear from commutativity.

For (6), by the Nullstellensatz any birational automorphism $g\in \Bir(X)$ in the kernel of the map \[
\Bir(X)\ra \Aut(X)
\]
is the identity on $U$, thus $g=\id\in \Bir(X)$.

Part (7) is proved by taking a tensor power of $\Lc$ that is very ample and applying (6).
\end{proof}

Now we shift our focus to giving examples of birationally equivariant line bundles.

\begin{proposition}\label{globgenbir}
Let $X$ be a smooth projective variety.
\begin{enumerate}
\item If $\omega_X$ is a globally generated line bundle then it is birationally equivariant.
\item More generally, if the image of the evaluation map:
\[
H^0(X,\mybigwedge^i \Omega_X)\otimes_k \Oc_X \ra \mybigwedge^i \Omega_X
\]
is a line bundle $\Lc\subset \mybigwedge^i \Omega_X$ (which is necessarily globally generated), then $\Lc$ is birationally equivariant.
\end{enumerate}
\end{proposition}

\begin{proof}
Part (2) implies part (1), so we will just prove part (2). Let $g\in \Bir(X)$ and let $i_g \cl U_g\hookrightarrow X$ denote the inclusion of the open set on which $g$ is defined (so the complement $X\setminus U_g$ has codimension $\ge 2$).

The derivative map
\[
\mybigwedge^i dg\cl g^*(\mybigwedge^i \Omega_X) \ra \mybigwedge^i\Omega_{U_g}
\]
pushes forward to an inclusion
\[
i_{g*}(\mybigwedge^i dg)\cl i_{g*} g^*(\mybigwedge^i \Omega_X) \ra \mybigwedge^i\Omega_{X}.
\]
This gives an injection on global sections:
\[
H^0(X,i_{g*} g^*(\mybigwedge^i \Omega_X))\hookrightarrow H^0(X,\mybigwedge^i\Omega_{X}).
\]
Now $i_{g*} g^*(\mybigwedge^i \Omega_X)$ contains the rank one subsheaf $g^*\Lc$. Moreover, every global section of $\Lc$ pulls back to a global section of $g^*\Lc$. So we have a commuting diagram of inclusions:
\[
\begin{tikzcd}
H^0(X,\Lc)\arrow[r,hook,"g^*"]\arrow[d,equals]& H^0(X,g^*\Lc)\arrow[d,hook]&H^0(X,\Lc)\arrow[d,equals]\\
H^0(X,\mybigwedge^i \Omega_X)\arrow[r,hook,"g^*"]&H^0(X,i_{g*} g^*(\mybigwedge^i \Omega_X))\arrow[r,hook,"\mybigwedge^i dg"]&H^0(X,\mybigwedge^i\Omega_X)
\end{tikzcd}
\]
As the spaces on the left and the right are of the same dimension, and the maps are all inclusions it follows that every map is an isomorphism. Lastly, the commutative diagram of evaluation maps
\[
\begin{tikzcd}
H^0(X,g^*\Lc)\otimes_k\Oc_X\arrow[d]\arrow[r,"\cong"]&H^0(X,\mybigwedge^i\Omega_X)\otimes_k \Oc_X\arrow[d]\\
g^*\Lc\arrow[hook,r,"\phi_g"]& \mybigwedge^i\Omega_X
\end{tikzcd}
\]
shows that the image of $g^*\Lc$ in $\mybigwedge^i\Omega_X$ contains $\Lc$. It remains to show that the image of $\phi_g$ contains nothing more (i.e. it is an isomorphism).

Suppose for contradiction that there was a sequence:
\[
0\ra \Lc \ra \phi_{g}(g^*\Lc)\ra \Oc_\Delta(\Delta)\ra 0,
\]
for some effective divisor $\Delta\subset X$. As $H^0(X,\Lc) \cong H^0(X,g^*\Lc)$ it follows that the fixed component of the linear series $|g^*\Lc|$ equals $\Delta$. However, the sections in $H^0(X,g^*\Lc)$ globally generate $g^*\Lc$ on $U_g$, so there is no fixed component. Therefore $\phi_g$ defines an isomorphism between $g^*\Lc$ and $\Lc$.

To check equality of isomorphisms
\[
\phi_{g_2}\circ g_2^*(\phi_{g_1}) = \phi_{g_1\cdot g_2}
\]
it suffices to check on any nonempty open set (as global automorphisms of a line bundle on a projective variety are constant). This reduces to the chain rule:
\[
\mybigwedge^i dg_2\circ g_2^* (\mybigwedge^i dg_1) = \mybigwedge^i d(g_1 \cdot g_2)
\]
on an open set where everything is defined.
\end{proof}

\section{ \texorpdfstring{$p$}{\texttwoinferior}-cyclic covers and their resolutions}

The goal of this section is to define $p$-cyclic covers in characteristic $p$, present Koll\'ar's resolution (\cite[\S 21]{Kollar95}), and check that they have terminal singularities (by further passing to a log resolution, and computing discrepancies). Throughout we work over an algebraically closed field $k$ of characteristic $p>0$.

First we define cyclic covers. Fix a $k$-scheme $X$ together with a line bundle $\Lc$ on $X$. Let
\[
\LL=\Spec_{\Oc_X}\left(\bigoplus_{i\ge 0} \Lc^{-i}\cdot y^i\right)
\]
(resp. $\LL^{\otimes m}$) be the total space of the line bundle $\Lc$ (resp. $\Lc^{\otimes m}$). Let $s\in H^0(X,\Lc^{\otimes m})$ be a section, which corresponds to a map
\begin{center}
\begin{tikzcd}
\LL^{\otimes m}\arrow[r]&X.\arrow[l,bend left=40,"s"]
\end{tikzcd}
\end{center}
There is also an $m$th power map: $\LL\xrightarrow{p_m}\LL^{\otimes m}$ which is a $\mu_m$-quotient.

\begin{definition}
The \textit{$m$-cyclic cover of $s$} is defined as $Y := p_m^{-1}(s(X))$. We say that the cyclic cover $Y$ has \textit{branch divisor} $(s=0)\subset X$.
\end{definition}

\noindent It follows that $Y\cong \Spec_{\Oc_X}\left(\bigoplus_{i\ge 0} \Lc^{-i} \cdot y^i/(y^m-s)\right)$.

Let $X$ be a smooth projective $k$-variety with a line bundle $\Lc$ and let $s\in H^0(X,\Lc^{\otimes p})$ be a section. The $p$-cyclic cover $Y$ of $s$ is typically singular. However, if $s$ is general then Koll\'ar shows how to resolve these singularities. We say $s$ has \textit{non-degenerate critical points} (\cite[17.3]{Kollar96}) if when we locally describe $s$ as a function, any critical point of $s$ has a non-degenerate Hessian matrix (when the characteristic is 2 this forces the dimension to be even, we leave out the odd-dimensional case here). In this case, any critical point of $s$ gives rise to to an isolated hypersurface singularity on $Y$ of the form:
\[
y^p=f_2(\xs)+f_3,
\]
where $f_2$ and $f_3$ are functions on $X$, $f_3$ vanishes to order 3, and $f_2(\xs)$ is a quadratic polynomial with non-degenerate Hessian. Koll\'ar shows that these isolated singularities can be resolved by a sequence of blow-ups at points.

If $p=2$ then $Y$ is resolved after one blow-up of each singular point, and this is a log-resolution (the exceptional divisor over each point is given by the quadric:
\[
y^2-f_2(\xs)=0\subset \PP^n,
\]
which can be checked to be smooth).

If $p>2$ then Koll\'ar shows that a sequence of $(p-1)/2$ blow-ups of isolated double points resolves the singularities of $Y$. At the $i$th step, the new exceptional divisor over $Y$ is the quadric in the new exceptional divisor over $\LL$:
\[
f_2(\xs)=0\subset \PP^n,
\]
here $\PP^n$ has coordinates $[x_1:\cdots:x_n:y]$. This is smooth and simple normal crossing away from the point $[0:\cdots:0:1]$. (The only exceptional divisor it intersects is the one from the step before, and the intersection is given by $(y=0)\cap (f_2=0)\subset\PP^n$, which is smooth.) Here the strict transform of $Y$ has the new local equation:
\[
y^{p-2i}=f_2(\xs)+f_3.
\]
Thus, it is resolved after $(p-1)/2$ steps. To give a log resolution, we must blow-up one more time at $[0:\cdots:0:1]$ in the final exceptional divisor. The last exceptional divisor of $Y$ is a smooth projective space $\PP^{n-1}$ with coordinates $[x_1:\cdots:x_n]$ and the intersection of the last 2 exceptional divisors is again the quadric $f_2(\xs) = 0\subset \PP^{n-1}$.

Call this log-resolution $Z$. This gives a log resolution of $Y$ which fits into the diagram
\begin{center}
\begin{tikzcd}
Z \arrow[d,swap,"\sigma"] \arrow[r, phantom, "\subset"] &[-0.2in] \Lt \arrow[d, swap, "\pi"] \arrow[dr, bend left, "\mut"] \\
Y \arrow[r, phantom, "\subset"] & \LL \arrow[r, swap, "\mu"] & X
\end{tikzcd}
\end{center}

\begin{proposition}\label{prop:termSing}
Let $X$ be a smooth $k$-variety of dimension $n\ge 3$ with a line bundle $\Lc$. If
\[
s\in H^0(X,\Lc^{\otimes p})
\]
is a section with non-degenerate critical points then the $p$-cyclic cover of $s$ has terminal singularities.
\end{proposition}
\begin{proof}
This can be checked locally at each singularity of the form
\[
y^p+f_2(\xs)+f_3.
\]

First, when the characteristic of $k$ is 2 (with $n$ even), if $\sigma\cl Z\ra Y$ is the resolution of singularities and $E$ is the unique exceptional divisor it suffices to compute the coefficient $\alpha$ of $E$ in the equation:
\[
K_Z = \sigma^*(K_Y)+\alpha E =  \mut^*(K_X+Y)|_Z+(n-2)E.
\]
If $n\ge 3$ then $\alpha>0$.

When $p$ is odd, let $E_i\subset Z$ denote the strict transform of the exceptional divisor of the $i$th blow-up of $T$. Let $r=(p-1)/2$. Then it suffices to compute the coefficients $\alpha_i$ of $E_i$:
\begin{align*}
K_Z = & \sigma^*(K_Y) +\alpha_1E_1+\cdots+\alpha_{r}
E_{r}+\alpha_{r+1}E_{r+1}\\
= & \pi^*(K_\LL+Y)|_Z+(n-2)E_1+(2n-4)E_2+\cdots+(rn-2r)E_r+((r+1)n-p)E_{r+1}.
\end{align*}
This gives $\alpha_i = i(n-2)$ for $0\le i\le r$ and $\alpha_{r+1} = (r+1)n-p$. These are all positive for $n\ge 3$.
\end{proof}

\section{Computing the space of \texorpdfstring{$(n-1)$}{\texttwoinferior}-forms}

Again, assume $k$ has characteristic $p>0$. In this section we prove a slightly more general version of Theorem B. Specifically we consider the following situation:
\begin{enumerate}
\item $X$ is a smooth projective $k$-variety of dimension $n\ge 3$,
\item $\Lc$ is an effective line bundle, $s\in H^0(X,\Lc^{\otimes p})$ is a global section with nondegenerate critical points,
\item $Y\subset \LL\xrightarrow{\mu} X$ is the $p$-cyclic cover of $s$,
\item and assume that $-K_Y$ is ample (i.e. $Y$ is Fano).
\end{enumerate}
In Proposition 3.2, assuming that
\[
H^0(X,T_X\otimes\omega_X\otimes \Lc^{p-1})=0
\]
we show
\[
H^0(Y,\mybigwedge^{n-1}\Omega_{Z}) = H^0(X,\omega_X\otimes \Lc^p).
\]
As a consequence, we show the line bundle $\mu^*(\omega_X\otimes \Lc^p)|_Y$ is birationally equivariant on $Y$. Theorems A and B and Corollary C follow from results in \S1.

Let $Z\subset \Lt$ be the log-resolution of the cyclic cover as in \S2. Following Koll\'ar, consider the relative cotangent sequence for $\Lt/X$ restricted to $Z$ and the cotangent sequence for $Z\subset \Lt$. This gives rise to the diagram:
\[
\begin{tikzcd}
&&I/I^2\arrow[d,hook]\arrow[r]&\tau\arrow[d]&\\
0\arrow[r]&\mut^*\Omega_X|_Z\arrow[r]\arrow[d]&\Omega_{\Lt}|_{Z}\arrow[r,"\rho_1"]\arrow[d]&\Omega_{\Lt/X}|_{Z}\arrow[r]\arrow[d]&0\\
0\arrow[r]&Q\arrow[r]&\Omega_{Z}\arrow[r,"\rho_2"]&B\arrow[r]&0.
\end{tikzcd}
\]
Here, $\tau$ (resp. $B$) is the torsion (resp. torsion-free) part of $\Omega_{\Lt/X}|_Z$. To check that $\rho_1$ extends to a map $\rho_2$, it suffices to check that $I/I^2$ maps to $0$ in $B$ which can be done generically as these are both torsion-free. Generically, this follows from the fact that the $y$-derivative of the equation of the cyclic cover:
\[
y^p-s(x_1,\dots,x_n)=0
\]
vanishes (as we are in characteristic $p$). $Q$ is defined to be the kernel of $\rho_2$.

\begin{proposition} In the above setting:
\begin{enumerate}
\item The natural map
\[
\mybigwedge^{n-1}\Omega_Z\cong \Omega_Z^{\vee}\otimes \omega_Z\ra Q^\vee\otimes \omega_Z
\]
is surjective outside of codimension 2.
\item The kernel is isomorphic to $\det(Q)$.
\item If $H^0(X,T_X\otimes\omega_X\otimes \Lc^{p-1})=0$ then $H^0(Z,\mybigwedge^{n-1}\Omega_Z) \cong H^0(Z,\det(Q))$.
\end{enumerate}
\end{proposition}

\begin{proof}
For (1), the map
\[
\Omega_Z^{\vee}\otimes \omega_Z\ra Q^\vee\otimes \omega_Z
\]
only fails to be surjective on the locus where $B$ is not locally free (which has codimension $\ge 2$ as $B$ is torsion free). Letting $A$ denote the kernel of the map above, we observe that $A$ is a second syzygy sheaf, so it is reflexive. $A$ is therefore a line bundle as it is rank 1.

Now that we know the kernel is a line bundle, (2) can be verified outside codimension 2 where the sequence
\begin{equation}
0\ra A \ra \Omega_Z^{\vee}\otimes \omega_Z\ra Q^\vee\otimes \omega_Z
\end{equation}
becomes exact. So we have
\begin{align*}
c_1(A)+c_1(Q^\vee\otimes \omega_Z) = c_1(\mybigwedge^{n-1}\Omega_Z),
\end{align*}
which gives $c_1(A) = c_1(Q)$, i.e.  $A \cong \det Q$.

It remains to check (3). By equation (3.1) and part (2) it suffices to show that $Q^\vee\otimes \omega_Z$ has no global sections. There is an inclusion:
\[
Q^\vee\otimes \omega_Z \subset (\mut^*T_X)|_Z\otimes \omega_Z,
\]
so it suffices to show
\[
H^0(Z,(\mut^*T_X)|_Z\otimes \omega_Z)=0
\]

As $Y$ has terminal singularities there is an exact sequence:
\[
0\ra \sigma^*\omega_Y\ra \omega_Z \ra \Oc_\Delta(\Delta)\ra 0
\]
(where $\Delta$ is an effective exceptional divisor in $Z$). Pushing forward to $Y$ shows $\omega\cong \sigma_*\omega_Z$ (as they are isomorphic outside of points, torsion free, and the first is a line bundle). Then by the projection formula:
\[
\mu^*T_X|_Y\otimes \omega_Y\ra \mu^*T_X|_Y\otimes \sigma_*(\omega_Z)
\]
is an isomorphism. Therefore, as pushforward preserves global sections it suffices to show:
\[
H^0(Y,\mu^*T_X|_Y\otimes \omega_Y)=0.
\]
Now $\omega_Y = \omega_X\otimes \Lc^{p-1}$. As $Y$ is a $p$-cyclic cover,
\[
\mu_*(\Oc_Y)\cong \bigoplus_{i=0}^{p-1} \Lc^{-i}.
\]
Pushing forward $\mu^*T_X|_Y \otimes \omega_Y$ gives:
\[
H^0(Y,\mu^*T_X|_Y\otimes \omega_Y) = \bigoplus_{i=0}^{p-1}H^0(X,T_X\otimes \omega_X\otimes \Lc^{i}).
\]
This vanishes by the assumptions that $H^0(X,T_X\otimes \omega_X\otimes \Lc^{p-1})=0$ (here we use that $\Lc$ is effective to show the other summands vanish).
\end{proof}

\begin{proposition}\label{prop:detQ} Assume
\begin{enumerate}
\item[(a)] the vanishing condition from Prop. 3.1(3) holds,
\item[(b)] $\omega_X\otimes \Lc^{p}$ is globally generated, and
\item[(c)] $Y$ is Fano (or that $H^{0}(Y, \omega_{Y}) = 0$).
\end{enumerate}
Then:
\[
\det(Q)\cong \mut^*(\omega_X\otimes \Lc^p)(\Delta)
\]
for some effective divisor $\Delta$ that is exceptional for the birational map $\sigma$, and
\begin{equation}\label{isomForms}
H^0(Z,\mybigwedge^{n-1}\Omega_Z) \cong H^0(Z,\det(Q)) \cong H^0(Z,\mut^*(\omega_X\otimes \Lc^p)|_Z)\cong H^0(X,\omega_X\otimes \Lc^p).
\end{equation}
\end{proposition}

\begin{proof}
By Koll\'ar's work (\cite[\S23]{Kollar95}) there is an injection:
\begin{equation}
\mut^*(\omega_X\otimes \Lc^p)|_Z \hookrightarrow \mybigwedge^{n-1}\Omega_Z.
\end{equation}
The line bundle $\mut^*(\omega_X\otimes \Lc^p)$ is globally generated, so it must land inside of $\det(Q)$. Hence, they fit into a short exact sequence:
\[
0\ra \mut^*(\omega_X\otimes \Lc^p)|_Z \ra \det(Q) \ra \Oc_\Delta(\Delta)\ra 0,
\]
for some effective divisor $\Delta$. It is not hard to see that $\det(Q)$ and $\mut^*(\omega_X\otimes \Lc^p)|_Z$ are isomorphic away from the exceptional divisors (from the definition of $Q$ and by how Koll\'ar defines the injection). So $\Delta$ is exceptional for $\sigma$.

Pushing forward along $\sigma$ gives a map on $Y$:
\[
\mu^*(\omega_X\otimes \Lc^p)|_Y \ra \sigma_*(\det(Q)).
\]
which is necessarily an isomorphism, as $\sigma_*(\det(Q))$ is torsion free and they are isomorphic away from points. It follows that
\[
H^0(Z,\det(Q))\cong H^0(Z,\mut^*(\omega_X\otimes \Lc^p)|_Z)\cong H^0(Y,\mu^*(\omega_X\otimes \Lc^p)|_Y).
\]
Lastly, we have:
\begin{equation}
\mu_*(\mu^*(\omega_X\otimes \Lc^p)|_Y) \cong \bigoplus_{i=0}^{p-1} \left(\omega_X \otimes \Lc^{p-i}\right).
\end{equation}
By the Fano assumption, $H^{0}(Y, \omega_{Y}) = 0$. We have $\omega_{Y} = \mu^{\ast}(\omega_{X} \otimes \Lc^{p-1})$, thus:
\[
\mu_{\ast}(\omega_{Y}) = \bigoplus_{i=0}^{p-1} \left(\omega_X \otimes \Lc^{p-1-i}\right)
\]
has no global sections. It follows that the only global sections on the right hand side of (3.4) come from $\omega_X \otimes \Lc^p$, giving:
\[
H^0(Y,\mu^*(\omega_X\otimes \Lc^p)|_Y)\cong H^0(X,\omega_X\otimes \Lc^p),
\]
which completes the proof.
\end{proof}

\begin{proof}[Proof of Theorem B]

We will check that in the setting of Theorem $B$, the assumptions of Proposition~\ref{prop:detQ} are satisfied. Let $e$ be a positive integer such that
\[
e+(p-1)d \leq n\le e+pd-3.
\]
Consider a hypersurface $X \subset \PP^{n+1}_{k}$ of degree $e$ and let $\Lc = \Oc_{X}(d)$ for some $d \geq 1$. We claim that
\[ H^{0}\left(X, T_{X} \otimes \omega_{X} \otimes \Lc^{p-1} \right) = H^{0}\left(X, T_{X}(e + (p-1)d - n - 2) \right) = 0. \]
The Euler sequence restricted to $X$
\[
0 \rightarrow \Oc_{X}(e+(p-1)d-n-2) \rightarrow \Oc_{X}^{\oplus(n+2)}(e+(p-1)d-n-1) \rightarrow T_{\PP^{n+1}}(e+(p-1)d-n-2)|_X \rightarrow 0
\]
can be used to show that $T_{\PP^{n+1}}(e+(p-1)d-n-2)|_X$ has no global sections. So the above vanishing follows from taking global sections for the inclusion of tangent bundles:
\[
T_{X}(e+(p-1)d-n-2) \hookrightarrow T_{\PP^{n+1}}(e+(p-1)d-n-2) \big|_{X}.
\]
Next, observe that
\[
\omega_{X} \otimes \Lc^{p} \cong \Oc_{X}(e+pd-n-2)
\]
so the inequality $e+pd-3\ge n$ implies that $\omega_{X} \otimes \Lc^{p}$ is globally generated.

By \cite[V.5.7.2]{Kollar96}, a general section $s\in H^0(X,\Oc_X(pd))$ has nondegenerate critical points, and $Y$ is Fano as $\omega_Y = \mu^*(\Oc_X(e+(p-1)d-n-2))|_Y.$ By Proposition~\ref{prop:detQ}, it follows that
\[ H^{0}(Z, \mybigwedge^{n-1}\Omega_{Z}) \cong H^{0}(X, \omega_{X}(pd)). \qedhere \]
\end{proof}

\begin{proof}[Proof of Corollary C]
By \S 2, $Y$ admits a log resolution $\mut \colon Z \rightarrow Y$ and there is an injection
\begin{equation}
\mut^*(\omega_X(pd))|_Z \hookrightarrow \mybigwedge^{n-1}\Omega_Z.
\end{equation}
which induces the isomorphism on global sections in \eqref{isomForms}. In particular, the image of the evaluation map
\[
H^{0}(Z, \mybigwedge^{n-1}\Omega_{Z}) \otimes \Oc_{Z} \rightarrow \mybigwedge^{n-1}\Omega_{Z}
\]
is precisely the line bundle
\[
\Lc := \mut^*(\omega_X(pd))|_Z,
\]
so by Proposition~\ref{globgenbir}(2) it is birationally equivariant on $Z$. Theorem~\ref{PropertiesBirThm}(3) shows that $\mu^{\ast}(\omega_X(pd))|_{Y}$ is birationally equivariant on $Y$. The map $\mu$ is purely inseparable and the sections of $\mu^{\ast}(\omega_X(pd))|_{Y}$ define a map
\[ Y \rightarrow X \subset \PP(H^{0}(Y, \mu^{\ast}(\omega_X(pd))|_{Y})^{\vee}). \]
Therefore, Theorem~\ref{PropertiesBirThm}(6)--(7) implies that $\Bir(Y) \cong \Aut(Y) \hookrightarrow \Aut(X)$.
\end{proof}

We are now ready to give:

\begin{proof}[Proof of Theorem A]

Let $X$ be a general degree $e\ge 3$ hypersurface over an algebraically closed field of characteristic $p>0$ and fix the line bundle $\Lc := \Oc_{X}(d)$. Assume (as in Corollary C) that $n\ge 3$ and
\[
e+(p-1)d\le n\le e+pd-3.
\]
Let $Y$ be a cyclic cover over a general section of $H^0(X,\Oc_X(pd))$. By Corollary C,
\[
\Bir(Y) \hookrightarrow \Aut(X).
\]
By work of Matsumura and Monsky \cite{MM63} (see \cite[Cor. 1.9]{Poonen05} for a more modern treatment), we may assume that $\Aut(X) = \{ 1 \}$. The index of such a $Y$ is $n+2-e-(p-1)d$. For appropriate choices of $e$ and $n$ this can be made arbitrarily large. For example, $p=2$, $e=3$ and $d=3$ give index 2 examples. When $p=2$, $e=3$, $d=4$ and $n=8$ there are index 3 examples.
\end{proof}

\bibliographystyle{siam}
\bibliography{Biblio.bib}

\begin{thebibliography}{10}

\bibitem{Cheltsov2000}
{\sc I.~Cheltsov}, {\em On a smooth four-dimensional quintic}, Mat. Sb, 191
  (2000), pp.~139--160.

\bibitem{FanoIrrat}
{\sc N.~Chen and D.~Stapleton}, {\em Fano hypersurfaces with arbitrarily large
  degrees of irrationality}, Forum of Mathematics, Sigma, 8 (2020), pp.~e.24,
  12pp.

\bibitem{RatEnd}
\leavevmode\vrule height 2pt depth -1.6pt width 23pt, {\em Rational
  endomorphisms of {F}ano hypersurfaces}, 2021.

\bibitem{Corti95}
{\sc A.~Corti}, {\em Factoring birational maps of 3-folds after {S}arkisov}, J.
  Algebraic Geom., 4 (1995), pp.~223--254.

\bibitem{deFernex13}
{\sc T.~de~Fernex}, {\em Birationally rigid hypersurfaces}, Inventiones
  mathematicae, 192 (2013), pp.~533--566.

\bibitem{deFernex16Erratum}
\leavevmode\vrule height 2pt depth -1.6pt width 23pt, {\em Erratum to:
  Birationally rigid hypersurfaces}, Inventiones mathematicae, 203 (2016),
  pp.~675--680.

\bibitem{dFEM03}
{\sc T.~de~Fernex, L.~Ein, and M.~Musta\c{t}\u{a}}, {\em Bounds for log
  canonical thresholds with applications to birational rigidity}, Math. Res.
  Lett., 10 (2003), pp.~219--236.

\bibitem{Fano1908}
{\sc G.~Fano}, {\em Sopra alcune varieta algebriche a tre dimensioni aventi
  tutti i generi nulli}, Atti. Ac. Torino, 43 (1908), pp.~973--977.

\bibitem{Fano1915}
\leavevmode\vrule height 2pt depth -1.6pt width 23pt, {\em Osservazioni sopra
  alcune varieta non razionali aventi tutti i generi nulli}, Atti. Ac. Torino,
  50 (1915), pp.~1067--1072.

\bibitem{IskMan71}
{\sc V.~A. Iskovskikh and J.~I. Manin}, {\em Three-dimensional quartics and
  counterexamples to the {L}\"{u}roth problem}, Mat. Sb., 86(128) (1971),
  pp.~140--166.

\bibitem{Kollar95}
{\sc J.~Koll\'{a}r}, {\em Nonrational hypersurfaces}, J. Amer. Math. Soc., 8
  (1995), pp.~241--249.

\bibitem{Kollar96}
\leavevmode\vrule height 2pt depth -1.6pt width 23pt, {\em Rational curves on
  algebraic varieties}, vol.~32 of Ergebnisse der Mathematik und ihrer
  Grenzgebiete, Springer-Verlag, Berlin, 1996.

\bibitem{MM63}
{\sc H.~Matsumura and P.~Monsky}, {\em On the automorphisms of hypersurfaces},
  Journal of Mathematics of Kyoto University, 3 (1963), pp.~347--361.

\bibitem{Poonen05}
{\sc B.~Poonen}, {\em Varieties without extra automorphisms {III}:
  hypersurfaces}, Finite fields and their applications, 11 (2005),
  pp.~230--268.

\bibitem{Pukhlikov87}
{\sc A.~V. Pukhlikov}, {\em Birational isomorphisms of four-dimensional
  quintics}, Inventiones mathematicae, 87 (1987), pp.~303--329.

\bibitem{Pukhlikov02}
\leavevmode\vrule height 2pt depth -1.6pt width 23pt, {\em Birationally rigid
  {F}ano hypersurfaces}, Izvestiya: Mathematics, 66 (2002), p.~1243.

\bibitem{Pukhlikov10}
\leavevmode\vrule height 2pt depth -1.6pt width 23pt, {\em Birational geometry
  of {F}ano double spaces of index two}, Izvestiya: Mathematics, 74 (2010),
  pp.~925--991.

\bibitem{Pukhlikov16}
\leavevmode\vrule height 2pt depth -1.6pt width 23pt, {\em Birational geometry
  of {F}ano hypersurfaces of index two}, Mathematische Annalen, 366 (2016),
  pp.~721--782.

\bibitem{Pukhlikov20}
\leavevmode\vrule height 2pt depth -1.6pt width 23pt, {\em Birational geometry
  of singular {F}ano hypersurfaces of index two}, Manuscripta Math., 161
  (2020), pp.~161--203.

\bibitem{Pukhlikov21}
\leavevmode\vrule height 2pt depth -1.6pt width 23pt, {\em Birational geometry
  of singular {F}ano double spaces of index two}, Sbornik: Mathematics, 212
  (2021), pp.~551--566.

\bibitem{Schreieder19}
{\sc S.~Schreieder}, {\em Stably irrational hypersurfaces of small slopes}, J.
  Amer. Math. Soc., 32 (2019), pp.~1171--1199.

\bibitem{Totaro}
{\sc B.~Totaro}, {\em Hypersurfaces that are not stably rational}, J. Amer.
  Math. Soc., 29 (2016), pp.~883--891.

\end{thebibliography}

\footnotesize{
\textsc{Department of Mathematics, Harvard University, Cambridge, Massachusetts 02138} \\
\indent \textit{E-mail address:} \href{mailto:nathanchen@math.harvard.edu}{nathanchen@math.harvard.edu}

\textsc{Department of Mathematics, University of Michigan, Ann Arbor, Michigan 48109} \\
\indent \textit{E-mail address:} \href{mailto:dajost@umich.edu}{dajost@umich.edu}
}

\end{document}